\documentclass[english, 12pt]{amsart}

\usepackage{tikz}
\usepackage{aliascnt}
\usepackage{mathrsfs}
\usepackage[all,poly,knot]{xy}
\usepackage{hyperref}
\usepackage{comment}  
\usepackage{csquotes}   
\usepackage{amssymb,amsbsy,amsmath,amsfonts,amssymb,amscd,
	graphics,color,footmisc,fancyhdr,multicol,fancybox,
	graphicx,mathrsfs,rotating,ifthen,wasysym}

\usepackage{times}
\usepackage{pdfpages}
\usepackage{multirow}
\usepackage{nccrules}
\usepackage{textcomp}
\usepackage{comment}
\usepackage{framed}
\usepackage[all]{xy}
\usepackage{graphicx}
\usepackage{pgf,tikz}
\usepackage{mathrsfs}
\usetikzlibrary{arrows}
\usepackage{lipsum}
\usepackage{mathtools}

\newcommand{\medcup}{\mathbin{\scalebox{1.5}{\ensuremath{\cup}}}}

\DeclareMathOperator{\dif}{\text{\normalfont d}}

\DeclareMathOperator{\Wron}{Wron}

\DeclareMathOperator{\Const}{Const.}

\DeclareMathOperator{\sing}{sing}

\DeclareMathOperator{\bs}{Bs}

\def\log{\mathrm{log}\,}

\theoremstyle{plain}

\newtheorem{thm}{Theorem}[section]  

\newtheorem{cor}[thm]{{Corollary}}

\newtheorem{pro}[thm]{Proposition}

\newtheorem{defi}[thm]{Definition}

\theoremstyle{remark}

\numberwithin{equation}{section}

\usepackage[top=1.5in, bottom=1.5in, left=1in, right=1in]{geometry}

\usepackage[hyperpageref]{backref}

\usepackage{xcolor}
\hypersetup{
	colorlinks,
	linkcolor={red!50!black},
	citecolor={blue!62!black},
	urlcolor={blue!80!black}
}
\theoremstyle{plain}
\newcommand{\thistheoremname}{}
\newtheorem*{genericthm*}{\thistheoremname}
\newenvironment{namedthm*}[1]{\renewcommand{\thistheoremname}{#1}%
	\begin{genericthm*}}
	{\end{genericthm*}}

\newtheoremstyle{named}{}{}{\itshape}{}{\bfseries}{.}{.5em}{\thmnote{#3's }#1}
\theoremstyle{named}

\makeatletter
\newcommand\thankssymb[1]{\textsuperscript{\@fnsymbol{#1}}}
\makeatother
\begin{document} 
	\title[On the Gauss maps of complete minimal surfaces]{\bf On the Gauss maps\\
		of complete minimal surfaces in $\mathbb{R}^n$}

	\subjclass[2020]{53A10, 32H30}
	\keywords{value distribution theory, Gauss map, minimal surface, hyperbolicity}
	
	\author{Dinh Tuan Huynh}
	
	\address{Department of Mathematics, University of Education, Hue University, 34 Le Loi St., Hue City, Vietnam}
	\email{dinhtuanhuynh@hueuni.edu.vn}
\dedicatory{Delicated to Professor Doan The Hieu}
\begin{abstract}
	We prove that the Gauss map of a non-flat complete  minimal surface immersed in $\mathbb{R}^n$ can omit a generic hypersurface $D$ of degree at most $
	n^{n+2}(n+1)^{n+2}$.
\end{abstract}

\maketitle

\section{Introduction}

Let $f = (x_1, x_2,\dots, x_n) \colon M\rightarrow \mathbb{R}^n$ be an oriented surface immersed in $\mathbb{R}^n$. Using systems of isothermal coordinates $(x,y)$, one can consider $M$ as a Riemann
surface. We are interested in the class of minimal surfaces, namely, those which
have minimal areas for all small perturbations. It is a well-known fact that if $M$ is minimal, then its Gauss map $g\colon M\rightarrow\mathbb{C}\mathbb{P}^{n-1}$, defined as

$$
g(z):=\bigg[\dfrac{\partial f}{
\partial z}\bigg],
$$
where $z=x+iy$ is a holomophic chart on $M$, is a holomorphic map.

In the  particular case where $n=3$, by identifying the unit sphere with the complex projective line via the stereographic projection, one can view the Gauss map $g$ as a map of $M$ into $\mathbb{C}\mathbb{P}^1$. Osserman \cite{Osserman64} proved that if $M$ is a non-flat complete minimal surface
	immersed in $\mathbb{R}^3$, then the complement of the image of its Gauss map is	of logarithmic capacity zero in $\mathbb{C}\mathbb{P}^1$. This interesting result  could be regarded as a significant improvement of the classical Bernstein's Theorem. Strengthening this result,  Xavier \cite{Xavier81} proved that in this situation, the Gauss map of $M$ can avoid at most $6$ points. Sharp result was obtained by Fujimoto  \cite{Fujimoto88}, where he proved that indeed, the Gauss map of $M$ can avoid at most $4$ points.

Passing to higher dimensional case, first step was made by Fujimoto \cite{Fujimoto83}, where the intersection between the Gauss maps  of a complete minimal surface immersed in $\mathbb{R}^n$ and family of hyperplanes in $\mathbb{C}\mathbb{P}^{n-1}$ was considered. Precisely, Fujimoto established the following
\begin{thm}
	\label{first high dim result of Fuj}
If the Gauss map of a non-flat complete minimal surface in $\mathbb{R}^n$ is non-degenerate, it can omit at most $q=n^2$ hyperplanes in $\mathbb{C}\mathbb{P}^{n-1}$ in general position.
\end{thm}
Later, Fujimoto himself  \cite{Fujimoto90} decreased the number of hyperplanes in the above statement to  $q=\dfrac{n(n+1)}{2}$ and it turns out that this number is sharp. Ru \cite{Ru91} was able to remove the non-degenerate assumption of the Gauss map in Fujimoto's result.
Since then, by adapting tools and techniques from value distribution theory of holomorphic curves to study Gauss maps, many generalizations of the above works of Fujimoto-Ru were made. Note that in these results, it is required the presence of many hypersurfaces.

In this paper, based on recent progresses towards the hyperbolicity problem \cite{Siu2015, Demailly2012, Demailly2020, DMR2010, Darondeau2016, HVX2019, Berczi2019, Brotbek-Deng2019}, we consider the case when there is only one hypersurface of  high enough degree.

\begin{namedthm*}{Main Theorem}
	Let $M$ be a non-flat complete minimal surface immersed in $\mathbb{R}^n$ and let $G\colon M\rightarrow \mathbb{C}\mathbb{P}^{n-1}$ be its Gauss map. Then $G$ can avoid a generic hypersurface $D\subset \mathbb{C}\mathbb{P}^{n-1}$ of  degree at most 
	\[
		d
	=
	n^{n+2}(n+1)^{n+2}.
	\] 
\end{namedthm*}
\section*{Acknowledgement}
This work is supported by the Vietnam Ministry of Education and Training under the grant number B2024-DHH-14. I would like to thank Prof. Doan The Hieu for his encouragements.  I want to thank Song-Yan Xie for helpful suggestions which improved the exposition.
\section{Logarithmic jet differentials}

Let $X$ be a complex projective variety of dimension $n$. For a point $x\in X$, consider the holomorphic germs $(\mathbb{C},0)\rightarrow (X,x)$. Two such germs are said to be equivalent if they have the same Taylor expansion up to order $k$ in some local coordinates around $x$. The equivalence class of an analytic germ $f\colon (\mathbb{C},0)\rightarrow (X,x)$ is called the {\sl $k$-jet of $f$}, denoted by $j_k(f)$, which is independent of the choice of local coordinates. A $k$-jet $j_k(f)$ is said to be {\sl regular} if $\dif f(0)\not=0$. For a given point $x\in X$, denote by $j_k(X)_x$ the vector space of all $k$-jets of analytic germs $(\mathbb{C},0)\rightarrow (X,x)$, set
\[
J_k(X)
:=
\underset{x\in X}{\medcup}\,J_k(X)_x,
\]
and consider the natural projection
\[
\pi_k\colon J_k(X)\rightarrow X.
\]
Then $J_k(X)$ carries the structure of a holomorphic fiber bundle over $X$, which is called the {\sl $k$-jet bundle over $X$}. Note that in general, $J_k(X)$ is not a vector bundle. When $k=1$, the $1$-jet bundle $J_1(X)$ is canonically isomorphic to the tangent bundle $T_X$ of $X$.

For an open subset $U\subset X$, for a section $\omega\in H^0(U,T_X^*)$, for a $k$-jet $j_k(f)\in J_k(X)|_U$, the pullback $f^*\omega$ is of the form $A(z)\dif z$ for some analytic function $A$, where $z$ is the global coordinate of $\mathbb{C}$. Since each derivative $A^{(j)}$ ($0\leq j\leq k-1$) is well-defined, independent of the representation of $f$ in the class $j_k(f)$, the analytic $1$-form $\omega$ induces the holomorphic map
\begin{equation}
\label{trivialization-jet}
\tilde{\omega}
\colon
J_k(X)|_U
\rightarrow
\mathbb{C}^k;\,\,j_k(f)\rightarrow
\big(A(z),A(z)^{(1)},\dots,A(z)^{(k-1)}
\big).
\end{equation}
Hence on an open subset $U$, a given local holomorphic coframe $\omega_1\wedge\dots\wedge\omega_n\not=0$ yields a trivialization 
\[
H^0(U, J_k(X))\rightarrow U\times(\mathbb{C}^k)^n
\]
by providing the following new $n k$ independent coordinates:
\[
\sigma\rightarrow(\pi_k\circ\sigma;\tilde{\omega}_1\circ\sigma,\dots,\tilde{\omega}_n\circ\sigma),
\]
where $\tilde{\omega}_i$ are defined as in \eqref{trivialization-jet}. The components $x_i^{(j)}$ ($1\leq i\leq n$, $1\leq j\leq k$) of $\tilde{\omega}_i\circ\sigma$ are called the {\sl jet-coordinates}.
In a more general setting, where $\omega$ is a section over $U$ of the sheaf of meromorphic $1$-forms, the induced map $\tilde{\omega}$ is meromorphic.

Now, suppose that $D\subset X$ is a normal crossing divisor on $X$. This means that at each point $x\in X$, there exist some local coordinates $z_1,\dots,z_{\ell},z_{\ell+1},\dots,z_n$ ($\ell=\ell(x)$) centered at $x$ in which $D$ is defined by
\[
D=
\{z_1\dots z_{\ell}=0
\}.
\]
Following Iitaka \cite{Iitaka1982}, the {\sl logarithmic cotangent bundle of $X$ along $D$}, denoted by $T_X^*(\log D)$, corresponds to the locally free sheaf generated by
\[
\dfrac{\dif\!z_1}{z_1},\dots,\dfrac{\dif\! z_{\ell}}{z_{\ell}},z_{\ell +1},\dots,z_n
\]
in the above local coordinates around $x$.

A holomorphic section $s\in H^0(U,J_k(X))$ over an open subset $U\subset X$ is said to be a {\sl logarithmic $k$-jet field} if $\tilde{\omega}\circ s$ are analytic for all sections $\omega\in H^0(U',T_X^*(\log D))$, for all open subsets $U'\subset U$, where $\tilde{\omega}$ are induced maps defined as in \eqref{trivialization-jet}. Such logarithmic $k$-jet fields define a
subsheaf of $J_k(X)$, and this subsheaf is itself a sheaf of sections of a holomorphic fiber bundle over $X$, called the {\sl logarithmic $k$-jet bundle over $X$ along $D$}, denoted by $J_k(X,-\log D)$ (see \cite{Noguchi1986}).

The group $\mathbb{C}^*$ admits a natural fiberwise action defined as follows. For local coordinates 
\[
z_1,\dots,z_{\ell},z_{\ell+1},\dots,z_n\eqno\scriptstyle{(\ell=\ell(x))}
\]
centered at $x$ in which $D=\{z_1\dots z_{\ell}=0\}$, for any logarithmic $k$-jet field along $D$ represented by some germ $f=(f_1,\dots,f_n)$, if $\varphi_{\lambda}(z)=\lambda z$ is the homothety with ratio $\lambda\in \mathbb{C}^*$, the action is given by
\[
\begin{cases}
\big(\log(f_i\circ \varphi_{\lambda})\big)^{(j)}
=
\lambda^j
\big(\log f_i\big)^{(j)}\circ\varphi_{\lambda} &
\quad
\scriptstyle{(1\,\leq\,i\,\leq\,\ell),}
\\
\big(f_i\circ \varphi_{\lambda}\big)^{(j)}
\quad\quad\,=
\lambda^j f_i^{(j)}\circ\varphi_{\lambda}
&
\quad
\scriptstyle{(\ell+1\,\leq\,i\,\leq\,n).}
\end{cases}
\]

A {\sl logarithmic jet differential of {\sl order} $k$ and {\sl degree}} $m$ at a point $x\in X$ is a polynomial $Q(f^{(1)},\dots,f^{(k)})$ on the fiber over $x$ of $J_k(X,-\log D)$ enjoying weighted homogeneity:
\[
Q(j_k(f\circ\varphi_{\lambda}))
=
\lambda^m
Q(j_k(f))
\eqno\scriptstyle{(\lambda\,\in\,\mathbb{C}^*)}.
\]

Consider the symbols
\[
\dif^{j}\log z_i\eqno\scriptstyle{(1\,\leq\, j\,\leq\, k,\,1\,\leq\, i\,\leq\,\ell)}
\]
and
\[
\dif^{j}z_i\eqno\scriptstyle{(1\,\leq\, j\,\leq\, k,\,\ell\,+\,1\,\leq\, i\,\leq\,n)}.
\]
Set the weight of $\dif^{j}\log z_i$ or $\dif^{j}z_i$ to be $j$. Then
a logarithmic jet differential of order $k$ and weight $k$ along $D$ at $x$ is a weighted homogeneous polynomial of degree $m$ whose variables are these symbols. Denote by $E_{k,m}^{GG}T_X^*(\log D)_x$ be the vector space spanned by such polynomials and set
\[
E_{k,m}^{GG}T_X^*(\log D)
:=
\underset{x\in X}{\medcup}\,
E_{k,m}^{GG}T_X^*(\log D)_x.
\]
By Fa\`{a} di bruno's formula \cite{Constantine1996,Merker2015}, one can check that $E_{k,m}^{GG}T_X^*(\log D)$ carries the structure of a vector bundle over $X$, called {\sl logarithmic Green-Griffiths vector bundle} \cite{Green-Griffiths1980}. A global section of $E_{k,m}^{GG}T_X^*(\log D)$ is called a {\sl logarithmic jet differential} of order $k$ and weight $m$ along $D$. Locally, a logarithmic jet differential form can be written as

{\footnotesize
	\begin{equation}
	\label{local expression of log jet}
	\underset{|\alpha_1|+2|\alpha_2|+\dots+k|\alpha_k|=m}
	{\sum_{\alpha_1,\dots,\alpha_k\in\mathbb{N}^n}}
	A_{\alpha_1,\dots,\alpha_k}
	\bigg(
	\prod_{i=1}^{\ell}
	\big(
	\dif\log z_i
	\big)^{\alpha_{1,i}}
	\prod_{i=\ell+1}^{n}
	\big(
	\dif z_i
	\big)
	^{\alpha_{1,i}}
	\bigg)
	\dots
	\bigg(
	\prod_{i=1}^{\ell}
	\big(
	\dif^k\log z_i
	\big)^{\alpha_{k,i}}
	\prod_{i=\ell+1}^{n}
	\big(
	\dif^kz_i
	\big)
	^{\alpha_{k,i}}
	\bigg),
	\end{equation}
}
where
\[
\alpha_{\lambda}
=
(\alpha_{\lambda,1},\dots,\alpha_{\lambda,n})
\in\mathbb{N}^n
\eqno
\scriptstyle{(1\,\leq\,\lambda\,\leq\, k)}
\]
are multi-indices of length 
\[
|\alpha_{\lambda}|
=
\sum_{1\leq i\leq n}
\alpha_{\lambda,i},
\]
and where $A_{\alpha_1,\dots,\alpha_k}$ are locally defined holomorphic functions. 

Assigning the weight $s$ for $\frac{\dif^s z_i}{z_i}$, then one can rewritten $\dif^j\log z_i$ as an isobaric polynomial of weight $j$ of variables $\frac{\dif^s z_i}{z_i}\, (1\leq s\leq j)$ with integer coefficients, namely

$$
\dif^j\log z_i
=
\underset{\beta_1+2\beta_2+\dots+j\beta_j=j}
{\sum_{\beta=(\beta_1,\dots,\beta_j)\in\mathbb{N}^j}}
b_{j\beta}\bigg(\dfrac{\dif z_i}{z_i}\bigg)^{\beta_1}\dots\bigg(\dfrac{\dif^j z_i}{z_i}\bigg)^{\beta_j},
$$
where $b_{j\beta}\in\mathbb{Z}$. Conversely, one can also express $\frac{\dif^j z_i}{z_i}$ as an isobaric polynomial of weight $j$ of variables $\dif^s\log z_i$ ($1\leq s\leq j$) with integer coefficients \cite{Brotbek-Deng2019}. Thus one can also use the following trivialization of logarithmic jet differentials:

{\footnotesize
	\begin{equation}
	\label{local expression of log jet, second form}
	\underset{|\beta_1|+2|\beta_2|+\dots+k|\beta_k|=m}
	{\sum_{\beta_1,\dots,\beta_k\in\mathbb{N}^n}}
	B_{\beta_1,\dots,\beta_k}
	\bigg(
	\prod_{i=1}^{\ell}
	\big(
	\frac{\dif z_i}{z_i}\big)^{\beta_{1,i}}
	\prod_{i=\ell+1}^{n}
	\big(
	\dif z_i
	\big)
	^{\beta_{1,i}}
	\bigg)
	\dots
	\bigg(
	\prod_{i=1}^{\ell}
	\big(\dfrac{
	\dif^k z_i}{z_i}
	\big)^{\beta_{k,i}}
	\prod_{i=\ell+1}^{n}
	\big(
	\dif^kz_i
	\big)
	^{\beta_{k,i}}
	\bigg),
	\end{equation}
}
where
\[
\beta_{\lambda}
=
(\beta_{\lambda,1},\dots,\beta_{\lambda,n})
\in\mathbb{N}^n
\eqno
\scriptstyle{(1\,\leq\,\lambda\,\leq\, k)}
\]
are multi-indices of length 
\[
|\beta_{\lambda}|
=
\sum_{1\leq i\leq n}
\beta_{\lambda,i},
\]
and where $B_{\beta_1,\dots,\beta_k}$ are locally defined holomorphic functions. 

Demailly \cite{Dem97} refined the Green-Griffiths' theory and considered the sub-bundle $E_{k,m}T^*_X(\log D)$ of $E_{k,m}^{GG}T^*_X(\log D)$, whose sections are logarithmic jet differentials that are invariant under arbitrary reparametrization of the source $\mathbb{C}$. Let \[
(X,D,V)
\]
be a  {\sl log--direct manifold}, i.e.,  a triple consisting of a projective manifold $X$, a simple normal crossing divisor $D$ on $X$ and  a holomorphic sub-bundle $V$ of the logarithmic tangent bundle $T_X(-\log D)$. Starting with a log--direct manifold $(X_0,D_0,V_0):=(X,D,T_X(-\log D))$, one then defines $X_1:=\mathbb{P}(V_0)$ together with the natural projection $\pi_1:X_1\rightarrow X_0$. Setting $D_1:=\pi_1^*D_0$, so that $\pi_1$ becomes a log--morphism, and defines the sub-bundle $V_1\subset T_{X_1}(-\log D_1)$  as
\[
V_{1,(x,[v])}:=\{\xi\in T_{X_1,(x,[v])}(-\log D_1):\pi_*\xi\in \mathbb{C} \cdot v\},
\]
one obtains the log-direct manifold $(X_1,D_1,V_1)$ from the initial one.
Any germ of a holomorphic map $f:(\mathbb{C},0)\rightarrow (X\setminus D,x)$ can be lifted to $f^{[1]}: \mathbb{C}\rightarrow X_1\setminus D_1$. Inductively, one can construct on $X=X_0$ the {\sl Demailly-Semple tower}\,:
\[
(X_k,D_k,V_k)\rightarrow\dots \rightarrow(X_1,D_1,V_1)\rightarrow(X_0,D_0,V_0),
\] 
together with the projections $\pi_k:X_k\rightarrow X_0$. Denote by $\mathcal{O}_{X_k}(1)$ the tautological line bundle on $X_k$. Then the direct image $(\pi_k)_*\mathcal{O}_{X_k}(m)$ of $\mathcal{O}_{X_k}(m)=\mathcal{O}_{X_k}(1)^{\otimes m}$, denoted by $E_{k,m}T_X^*(\log D)$, is a locally free subsheaf of $E_{k,m}^{GG}T_X^*(\log D)$ generated by all polynomial operators in the derivatives up to order $k$, which are furthermore invariant under any change of parametrization $(\mathbb{C},0)\rightarrow (\mathbb{C},0)$. From the construction, one can immediately check the following
\begin{namedthm*}{Direct image formula}
	For any ample line bundle $\mathcal{A}$ on $X$, one has
	\begin{equation}
	\label{direct image formula}
	H^0\big(X,E_{k,m}T_X^*(\log D)\otimes\mathcal{A}^{-1}\big)\cong
	H^0\big(X_k,\mathcal{O}_{X_k}(m)\otimes\pi_k^*\mathcal{A}^{-1}\big).
	\end{equation}	
\end{namedthm*}

The  bundles $E_{k,m}^{GG}T_X^*(\log D)$, $E_{k,m}T_X^*(\log D)$ are fundamental tools in studying the degeneracy of holomorphic curves into $\mathbb{C}\setminus D$. By the fundamental vanishing theorem of entire curves \cite{Dem97, Siu2015}, for any ample line bundle $\mathcal{A}$ on $X$, a non-trivial global section of $E_{k,m}^{GG}T_X^*(\log D)\otimes\mathcal{A}^{-1}$ gives a corresponding algebraic differential equation that all entire holomorphic function $f\colon\mathbb{C}\rightarrow X\setminus D$ must satisfy. The existence of these sections was proved recently \cite{Merker2015, Demailly2012}, provided that the order of jet is high enough. However, despite many efforts, the problem of controlling the base locus of these bundles can be only handled under the condition that the degree of $D$ must be very large compared with the dimension of the variety \cite{DMR2010, Berczi2019, Demailly2020, Siu2015, Brotbek-Deng2019}.

Now we consider the case where $D$ is a generic hypersurface of degree $d$ in $\mathbb{C}\mathbb{P}^n$. To guarantee the existence of logarithmic jet differentials along $D$, we consider the order jet $k=n+1$ and put
$$
k'=\dfrac{k(k+1)}{2},\qquad \delta=(k+1)n+k.
$$
Fixing two positive integers  $\epsilon>0$ and $r>\delta^{k-1}k(\epsilon+k\delta)$. For a smooth hypersurface $D$, denote by $Y_k(D)$ the log-Demailly-Semple  $k$-jet tower associated to $\big(\mathbb{C}\mathbb{P}^n,D,T_{\mathbb{C}\mathbb{P}^n}(-\log D)\big)$. For a line bundle $L$ on $\mathcal{O}_{Y_k(D)}$, denote by $\bs\big(\mathcal{O}_{Y_k(D)}L\big)$ the base locus of the line bundle $L$. We will employ the following key result in \cite{Brotbek-Deng2019}.
\begin{pro}(\cite[Cor. 4.5]{Brotbek-Deng2019})
	\label{base locus bd}
There exist $\beta,\widetilde{\beta}\in\mathbb{N}$ such that for any $\alpha\geq 0$ and for any generic hypersurface $D\in\big|\mathcal{O}_{\mathbb{C}\mathbb{P}^n(1)}^{\epsilon+(r+k)\delta}\big|$, one has
$$
\bs\big(\mathcal{O}_{Y_k(D)}(\beta+\alpha\delta^{k-1}k')
\otimes
\pi_{0,k}^*
\mathcal{O}_{\mathbb{C}\mathbb{P}^n(1)}^{\widetilde{\beta}+\alpha(\delta^{k-1}k(\epsilon+k\delta)-r)}
\big)
\subset
Y_k(D)^{\sing}\cup\pi_{0,k}^{-1}(D).
$$
\end{pro}
Using this result, Brotbek-Deng confirmed the logarithmic Kobayashi conjecture in the case where the degree of $D$ is large enough. We extract from their proof the following
\begin{thm} \label{the_existence_jet}
	\label{construction-jet-differential} 
	Let $D\subset\mathbb{P}^n(\mathbb{C})$ be a generic smooth hypersurface in $\mathbb{P}^n(\mathbb{C})$ having large enough degree
	\[
	d
	\geq
(n+1)^{n+3}(n+2)^{n+3}.
	\]
	Let $f\colon\Delta\rightarrow \mathbb{C}\mathbb{P}^n$ be a non-constant  holomorphic disk. If $f(\Delta)\not\subset D$, then for jet order $k=n+1$, there exist some weighted degree $m$, vanishing order $\widetilde{m}$ with $\widetilde{m}>2 m$ and some global logarithmic jet differential
	\[
	\mathscr{P}
	\in
	H^0
	\big(
	\mathbb{C}\mathbb{P}^n,
	E_{k,m}^{GG}T_{\mathbb{C}\mathbb{P}^n}^*(\log D)
	\otimes
	\mathcal{O}_{\mathbb{C}\mathbb{P}^n}(1)^{-\widetilde{m}}
	\big)
	\]
	such that
	\begin{align}
	\label{eq_canPfjets}
	\mathscr{P}\big(j_k(f)
	\big)
	\not\equiv
	0.
	\end{align}
\end{thm}

\begin{proof}
	We follow the arguments in \cite[Cor. 4.9]{Brotbek-Deng2019}, with a slightly modification to get higher vanishing order. First, putting $
	r_0
	=
	2\delta^{k-1}k'
	+
	\delta^{k-1}(\delta+1)^2
	=
	\delta^{k-1}(\delta+1)(\delta+2)$. Since
$$
k(k+\delta-1+k\delta)<(\delta+1)^2,
$$
any integer number $d\geq (r_0+k)\delta+2\delta$ can be written as
$$
d
=
\epsilon
+
(r+k)\delta,
$$
where
$k\leq \epsilon\leq k+\delta-1$ and $r>2\delta^{k-1}k'+\delta^{k-1}k(\epsilon+k\delta)$. For such $d$, since
$$
\lim_{\alpha\rightarrow\infty}
\dfrac{\beta+\alpha\delta^{k-1}k'}{-\widetilde{\beta}-\alpha(\delta^{k-1}k(\epsilon+k\delta)-r)}
=
\dfrac{\alpha \delta^{k-1}k'}{r-\delta^{k-1}k(\epsilon+k\delta)}
<
\dfrac{1}{2},
$$
using Proposition~\ref{base locus bd}, for $\alpha\gg 1$ large enough, there exists some global logarithmic jet differential  $$\mathscr{P}
\in
H^0
\big(
\mathbb{C}\mathbb{P}^n,
E_{k,m}^{GG}T_{\mathbb{C}\mathbb{P}^n}^*(\log D)
\otimes
\mathcal{O}_{\mathbb{C}\mathbb{P}^n}(1)^{-\widetilde{m}}
\big)
$$ satisfying~\eqref{eq_canPfjets} with $m=\beta+\alpha\delta^{k-1}k',\widetilde{m}=-\widetilde{\beta}-\alpha(\delta^{k-1}k(\epsilon+k\delta)-r)$ and $\widetilde{m}>2m$. Hence it remains to giving a lower bound for $(r_0+k)\delta+2\delta$. This could be done by a straightforward computation:
\begin{align*}
(r_0+k)\delta+2\delta
&=
\big(\delta^{k-1}(\delta+1)(\delta+2)+k+2\big)\delta\\
&<
(n+1)^{n+3}(n+2)^{n+3}.
\end{align*}
\end{proof}

\section{Value distribution theory for holomorphic maps from unit disc into projective spaces}
Let $E=\sum_i\alpha_i\,a_i$ be a divisor on the unit disc $\Delta$ where $\alpha_i\geq 0$, $a_i\in\Delta$ and let
$k\in \mathbb{N}\cup\{\infty\}$. For each $0<t<1$, denote by $\Delta_t$ the disk $\{z\in\mathbb{C},|z|<t\}$. Summing the $k$-truncated degrees of the divisor on disks by
\[
n^{[k]}(t,E)
:=
\sum_{a_i\in\Delta_t}
\min
\,
\{k,\alpha_i\}
\eqno
{{\scriptstyle (0\,<\,t\,<\,1)},}
\]
the \textsl{truncated counting function at level} $k$ of $E$ is then defined by taking the logarithmic average
\[
N^{[k]}(r,E)
\,
:=
\,
\int_0^r \frac{n^{[k]}(t, E)}{t}\,\dif\! t
\eqno
{{\scriptstyle (0\,<\,r\,<1)}.}
\]
When $k=\infty$, we write $n(t,E)$, $N(r,E)$ instead of $n^{[\infty]}(t,E)$, $N^{[\infty]}(r,E)$. Let $f\colon\Delta\rightarrow \mathbb{C}\mathbb{P}^n$ be an entire curve having a reduced representation $f=[f_0:\cdots:f_n]$ in the homogeneous coordinates $[z_0:\cdots:z_n]$ of $\mathbb{C}\mathbb{P}^n$. Let $D=\{Q=0\}$ be a divisor in $\mathbb{C}\mathbb{P}^n$ defined by a homogeneous polynomial $Q\in\mathbb{C}[z_0,\dots,z_n]$ of degree $d\geq 1$. If $f(\Delta)\not\subset D$, we define the \textsl{truncated counting function} of $f$ with respect to $D$ as
\[
N_f^{[k]}(r,D)
\,
:=
\,
N^{[k]}\big(r,(Q\circ f)_0\big),
\]
where $(Q\circ f)_0$ denotes the zero divisor of $Q\circ f$.

The \textsl{proximity function} of $f$ for the divisor $D$ is defined as
\[
m_f(r,D)
\,
:=
\,
\int_0^{2\pi}
\log
\frac{\big\Vert f(re^{i\theta})\big\Vert^d\,
	\Vert Q\Vert}{\big|Q(f)(re^{i\theta})\big|}
\,
\frac{\dif\!\theta}{2\pi},
\]
where $\Vert Q\Vert$ is the maximum  absolute value of the coefficients of $Q$ and
\[
\big\Vert f(z)\big\Vert
\,
=
\,
\max
\,
\{|f_0(z)|,\ldots,|f_n(z)|\}.
\]
Since $\big|Q(f)\big|\leq \Vert Q\Vert\cdot\Vert f\Vert^d$, one has $m_f(r,D)\geq 0$.

Lastly, the \textsl{Cartan order function} of $f$ is defined by
\begin{align*}
T_f(r)
\,
:&=
\,
\frac{1}{2\pi}\int_0^{2\pi}
\log
\big\Vert f(re^{i\theta})\big\Vert \dif\!\theta.
\end{align*}

With the above notations, the Nevanlinna theory consists of two fundamental theorems (for  comprehensive presentations, see  \cite{Noguchi-Winkelmann2014,Ru2021}).

\begin{namedthm*}{First Main Theorem}\label{fmt} Let $f\colon\Delta\rightarrow \mathbb{P}^n(\mathbb{C})$ be a holomorphic curve and let $D$ be a hypersurface of degree $d$ in $\mathbb{C}\mathbb{P}^n$ such that $f(\Delta)\not\subset D$. Then for every $r>1$, the following holds
	\[
	m_f(r,D)
	+
	N_f(r,D)
	\,
	=
	\,
	d\,T_f(r)
	+
	O(1),
	\]
	whence
	\begin{equation}
	\label{-fmt-inequality}
	N_f(r,D)
	\,
	\leq
	\,
	d\,T_f(r)+O(1).
	\end{equation}
\end{namedthm*}

On the other side, in the harder part, so-called Second Main Theorem, one tries to bound the order function from above by some  sum of certain counting functions. Such types of results were given in several situations, and most of them were relied on the following key estimate.

 \begin{namedthm*}{Logarithmic Derivative Lemma}
	Let $g$ be a non-constant meromorphic function on the unit disc and let $k\geq 1$ be a positive integer number. Then for any $0<r<1$, the following estimate holds
	\[
	m_{\frac{g^{(k)}}{g}}
	(
	r
	)
	:=
	m_{\frac{g^{(k)}}{g}}
	(
	r,\infty
	)
	=
	O\bigg(\log\dfrac{1}{1-r}\bigg)+O(\log T_g(r))\qquad\parallel,
	\]
where the notation $\parallel$	means that the above estimate holds true for all $0<r<1$ outside a subset $E\subset (0,1)$ with $$\int_E\dfrac{dr}{1-r}<\infty.$$
\end{namedthm*}

\section{An application of the Logarithmic Derivative Lemma}

It is a well-known fact that the growth of the order function of an entire holomorphic curve could be used to determine its rationality. Replacing the source of the curve by the unit disc $\Delta$, one has the following
\begin{defi}
A holomorphic map $f\colon\Delta\rightarrow\mathbb{C}\mathbb{P}^n$ is said to be transcendental if
$$
\limsup_{r\rightarrow 1}\dfrac{T_f(r)}{\log\dfrac{1}{1-r}}=\infty.
$$
\end{defi}

\begin{thm}
	\label{non trancendental of f}
	Let $f\colon\Delta\rightarrow \mathbb{C}\mathbb{P}^n$ be a holomorphic map and $D\subset \mathbb{C}\mathbb{P}^n$ be a generic hypersurface having 
	 large enough degree:
	\[
	d
	\geq
	(n+1)^{n+3}(n+2)^{n+3}.
	\] If $f$ avoids $D$, then it is not transcendental.
\end{thm}

\begin{proof}
	Employing the logarithmic jet differentials supplied by Theorem~\ref{construction-jet-differential}, following the arguments as in \cite{HVX2019} and using the Logarithmic Derivative Lemma for meromorphic functions on unit disc, one gets
	$$
	T_f(r)
	\leq
	N_f^{[1]}(r,D)
	+
	O\bigg(\log\dfrac{1}{1-r}\bigg)
	+O(\log T_f(r))
	=
	O\bigg(\dfrac{1}{1-r}\bigg)
	+O(\log T_f(r))\quad\parallel,
	$$
	whence concludes the proof.
\end{proof}

We will also need the following results due to Fujimoto \cite{Fujimoto83}.

\begin{pro} (\cite[Pro. 2.5]{Fujimoto83})
	Let $\varphi$ be a nowhere zero holomorphic function on $\Delta$ which
	is not transcendental. Then, for each positive integer number $\lambda$, the following estimate holds
	$$
	\int_{0}^{2\pi}
	\bigg|
	\dfrac{
	d^{\lambda-1}}{d z^{\lambda-1}}
	\bigg(
	\dfrac{\varphi'}{\varphi}
	\bigg)(re^{i\theta})
	\bigg|
	d\theta
	\leq
	\dfrac{\Const}{(1-r)^{\lambda}}\log\dfrac{1}{1-r}\eqno\scriptstyle{(0\,<\,r\,<1)}.
	$$
\end{pro}

\begin{cor}(\cite[Lem. 3.4]{Fujimoto83})
	\label{application of logarithmic derivative lemma}
Let $\varphi_1,\dots,\varphi_n$ be  nowhere zero holomorphic functions on $\Delta$ which
are not transcendental. Then, for any $n$-tuple of positive integer numbers $(\lambda_1,\dots,\lambda_n)$ and for any positive real number $t$ with $tn<1$, the following estimate holds
$$
\int_{0}^{2\pi}
\bigg|
\prod_{j=1}^n
\bigg(
\dfrac{\varphi_j'}{\varphi_j}
\bigg)^{(\lambda_j-1)}(re^{i\theta})
\bigg|^t
d\theta
\leq
\dfrac{\Const}{(1-r)^{s}}
\bigg(
\log\dfrac{1}{1-r}
\bigg)^{s}\eqno\scriptstyle{(0\,<\,r\,<1)},
$$
where $s=t(\sum_{j=1}^n\lambda_j)$.
\end{cor}

\section{Proof of the Main result}
\begin{pro}
	\label{estimate jet and norm f}
	Let $D\subset\mathbb{C}\mathbb{P}^{n}$ be a generic smooth hypersurface of degree $d$ and let $f\colon\Delta\rightarrow\mathbb{C}\mathbb{P}^n\setminus D$ be a non-degenerate holomorphic curve. Suppose that there exists a global logarithmic jet differential
	\[
	\mathscr{P}
	\in
	H^0
	\big(
	\mathbb{P}^n(\mathbb{C}),
	E_{n,m}^{GG}T_{\mathbb{C}\mathbb{P}^n}^*(\log D)
	\otimes
	\mathcal{O}_{\mathbb{C}\mathbb{P}^n}(1)^{-\widetilde{m}}
	\big)
	\]
	such that
	\begin{align}
	\label{condition non vanishing jet}
	\mathscr{P}\big(j_n(f)
	\big)
	\not\equiv
	0.
	\end{align}
	Then, there exists a positive constant $K$ such that
	$$
	\int_{0}^{2\pi}\big|\mathscr{P}\big(j_n(f)
	\big)(re^{i\theta})\big|^{\frac{2}{\widetilde{m}}}\|f(re^{i\theta})\|^2d\theta
	\leq
	\dfrac{K}{(1-r)^{\frac{2m}{\widetilde{m}}}}
	\bigg(
	\log
	\dfrac{1}{1-r}
	\bigg)^{\frac{2m}{\widetilde{m}}}
	\eqno
	\scriptstyle{(0<r<1)}.
$$
\end{pro}

\begin{proof}
Let $s$ be the canonical section of the ample line bundle $\mathcal{E}:=\mathcal{O}_{\mathbb{C}\mathbb{P}^n}(1)$. Since $\mathscr{P}$ vanishes on $\mathcal{E}$  with vanishing order $\widetilde{m}$, in any local chat $U_{\alpha}$ of $\mathbb{C}\mathbb{P}^n$, one can represent $\mathscr{P}s^{\widetilde{m}}$ as an isobaric polynomial $\mathscr{P}_s^{\alpha}$ of weight $m$ of variables 
$$
\dfrac{d^{\lambda}u_{j,\lambda}^{\alpha}}{u_{j,\lambda}^{\alpha}}
\eqno\scriptstyle{(1\,\leq\, \lambda\,\leq\, k,\, 1\,\leq\,j\,\leq\,n)},
$$
with local holomorphic coefficients,
where $u_{j,\lambda}$ are rational functions on $\mathbb{C}\mathbb{P}^n$. Consequently, we get that
\begin{align*}
|\mathscr{P}(j_k(f))|
\cdot
\|f\|^{\widetilde{m}}&
\leq
\sum_{\alpha}
\bigg|
\mathscr{P}_s^{\alpha}
\bigg(
\dfrac{d^{\lambda}(u_{j,\lambda}^{\alpha}\circ f)}{u_{j,\lambda}^{\alpha}\circ f}
\bigg)
\bigg|.
\end{align*}

Since $0<\frac{2}{\widetilde{m}}<1$, using the elementary inequality
\[
(x_1+\dots+x_r)^{\frac{2}{\widetilde{m}}}
<
x_1^{\frac{2}{\widetilde{m}}}
+
\dots+
x_r^{\frac{2}{\widetilde{m}}}
\eqno
\scriptstyle{(x_i>0)},
\]
the above estimate yields
\[
\big|\mathscr{P}\big(j_n(f)
\big)(re^{i\theta})\big|^{\frac{2}{\widetilde{m}}}\|f(re^{i\theta})\|^2
<
\sum_{\alpha}
\bigg|
\mathscr{P}_s^{\alpha}
\bigg(
\dfrac{d^{\lambda}(u_{j,\lambda}^{\alpha}\circ f)}{u_{j,\lambda}^{\alpha}\circ f}
\bigg)
\bigg|^{\frac{2}{\widetilde{m}}}.
\]
Hence it suffices to prove
\[
\int_{0}^{2\pi}\bigg|
\mathscr{P}_s^{\frac{2}{\widetilde{m}}}
\bigg(
\dfrac{d^{\lambda}(u_{j,\lambda}^{\alpha}\circ f)}{u_{j,\lambda}^{\alpha}\circ f}
\bigg)
\bigg|^{\frac{2m}{\widetilde{m}}}
d\theta
\leq
\dfrac{\Const}{(1-r)^{\frac{2m}{\widetilde{m}}}}
\bigg(
\log
\dfrac{1}{1-r}
\bigg)^{\frac{2m}{\widetilde{m}}}
\eqno
\scriptstyle{(0<r<1)}.
\]
By assumption,  $f$ avoids $D$,  hence it is not transcendental by Theorem~\ref{non trancendental of f}. Since each function $u_{j,\lambda}^{\alpha}$ is rational, it follows that $u_{j,\lambda}^{\alpha}\circ f$ is also not transcendental. Now, observing that each term
$$
\dfrac{d^{\lambda}(u_{j,\lambda}^{\alpha}\circ f)}{u_{j,\lambda}^{\alpha}\circ f}
$$
can be represented as a polynomial
$\mathscr{P}_{j,\lambda}^{\alpha}$ of variables 
$$
\dfrac{(u_{j,\lambda}^{\alpha}\circ f)'}{u_{j,\lambda}^{\alpha}\circ f},\dots,
\bigg(
\dfrac{(u_{j,\lambda}^{\alpha}\circ f)'}{u_{j,\lambda}^{\alpha}\circ f}
\bigg)^{\lambda-1},$$
which is isobaric of weight $\lambda$,
using Corollary~\ref{application of logarithmic derivative lemma}, one immediately gets the desired result.

\end{proof}
We will also need the following result of Yau \cite{Yau76} in the sequence.
\begin{thm}[\cite{Yau76}]
	\label{Yau}
Let $M$ be a complete Riemann manifold equipped with a volume form $d\sigma$. Let $h$ be a non-negative and non-constant smooth function on $M$ such that $\Delta \log h=0$ almost everywhere.	Then $\int_Mh^pd\sigma=\infty$ for any $p>0$.
\end{thm}
Now we enter the details of the proof of the Main Theorem. Let $f$ be the conjugate of $G$, which is a holomorphic map. It suffices to prove that $f$ is constant. Suppose on the contrary that this is not the case. Let $\pi\colon\widetilde{M}\rightarrow M$ be the  universal covering of $M$. Then $\widetilde{M}$ is also considered as a minimal surface in $\mathbb{R}^n$. Hence without lost of generality, we may assume $M=\widetilde{M}$. Since there is no compact minimal surface in $\mathbb{R}^n$, it follows that $M$ is biholomorphic to either $\mathbb{C}$ or $\Delta$. Thus we may assume  $M=\mathbb{C}$ or $M=\Delta$. The first case was excluded by recent work towards Kobayashi's conjecture (cf.\cite{Brotbek-Deng2019}). Hence it suffices to work in the case where $M=\Delta$. The area form of the metric on $M$ induced from the flat metric on $\mathbb{R}^n$ is given by 
$$
d\sigma=
2\|f\|^2du\wedge dv.
$$
Let $\mathscr{P}$ be a global logarithmic jet differential supplied by Theorem~\ref{construction-jet-differential}. Then it is clear that $h=|\mathscr{P}(j_k(f)|\not\equiv 0$ and $\Delta\log h=0$ for any $z$ out side the the zero set of $h$. Since $\Delta$ is complete, simply connected and of non-positive curvature, it has the infinite area with respect to the metric induced from $\mathbb{R}^n$. Using Theorem~\ref*{Yau}, on obtains that
\begin{equation}
\label{first estimate via Yau theorem}
\int_{\Delta}
h^{\frac{2}{\widetilde{m}}}d\sigma=
\infty.
\end{equation}
On the other hand, using Proposition~\ref{estimate jet and norm f}, one has
\begin{align*}
\int_{\Delta}
h^{\frac{2}{\widetilde{m}}}d\sigma
&=
2\int_{\Delta}
h^{\frac{2}{\widetilde{m}}}\|f\|^2dudv\\
&=
2
\int_{0}^1rdr
\bigg(
\int_{0}^{2\pi}h(re^{i\theta})^{\frac{2}{\widetilde{m}}}\|f(re^{i\theta})\|^2d\theta
\bigg)\\
&\leq
K
\int_{0}^1
\dfrac{r}{(1-r)^{\frac{2m}{\widetilde{m}}}}
\bigg(
\log
\dfrac{1}{1-r}
\bigg)^{\frac{2m}{\widetilde{m}}}
dr.
\end{align*}
The last integral in the above estimate is finite since $2m<\widetilde{m}$. This contradicts \eqref{first estimate via Yau theorem}. Therefore, the map $f$ must be constant, whence concludes the proof of the Main Theorem.

\section{Some discussions}

Theorem~\ref{first high dim result of Fuj} can be recovered via the above jet method. Indeed, according to Siu \cite{Siu2015}, the Wronskian can be employed to build a suitable logarithmic jet differentials. Precisely, let us consider the inhomogeneous coordinates $x_1,x_2,\dots,x_n$ of $\mathbb{C}\mathbb{P}^n$. Let $\{H_i\}_{1\leq i\leq q}$ be the family of hyperplanes in general position in $\mathbb{C}\mathbb{P}^n$. For each $1\leq i\leq q$, denote by $F_i$ the linear form of variables $x_1,\dots,x_n$ defining the hyperplane $H_i$. Put
$$
\omega
=
\dfrac{\Wron(dx_1,\dots,dx_n)}{F_1\dots F_q},
$$
where $\Wron$ denotes the Wronskian. The point is that by the assumption of general position, at any point $x=(x_1,\dots,x_n)$, there exists a set $I=\{i_1,\dots,i_{n}\}$ having cardinality $n$ such that $F_j$ are nowhere zero in a neighborhood $U$ of $x$ for all $j\not\in I$. Locally on $U$, one can write $\omega$ as

$$
\omega
=
\Const
\dfrac{\Wron(d\,\log F_{i_1}(x),\dots,d\,\log F_{i_n}(x))}{\Pi_{j\not\in I} F_j(x)},
$$
and hence, $\omega$ gives rise to a logarithmic jet differentials along the divisor $\sum_{i=1}^{q}H_i$. The denominator $F_1\dots F_q$ in $\omega$ gives the vanishing order $q$ at the infinity
hyperplane, hence direct computation yields immediately that $\omega$ is of weight $m=\frac{n(n+1)}{2}$ and  vanishes on the infinity hyperplane with the vanishing order $\widetilde{m}=q-(n+1)$.

Finally, in view of the result of Fujimoto-Ru, one can expect that the optimal degree bound in the statement of our Main Theorem should be $\dfrac{n(n+1)}{2}$.
\begin{namedthm*}{Conjecture}
Let $M$ be a non-flat complete minimal surface in $\mathbb{R}^n$ and let $G\colon M\rightarrow \mathbb{C}\mathbb{P}^{n-1}$ be its Gauss map. Then $G$ could avoid a generic hypersurface $D\subset \mathbb{C}\mathbb{P}^{n-1}$ of  degree at most
\[
d
=
\dfrac{n(n+1)}{2}.
\] 
\end{namedthm*}

\begin{center}
	\bibliographystyle{plain}
	
\end{center}
\address
\end{document}